\documentclass[12pt]{amsart}
\usepackage{graphicx}
\usepackage{tikz}
\usepackage{amsmath}
\usepackage{caption}
\usepackage{float}

\usetikzlibrary{calc}
\usepackage[letterpaper,left=1in,right=1in]{geometry}
\theoremstyle{theorem}%
\newtheorem{theorem}{Theorem}[section]

\newtheorem{lemma}[theorem]{Lemma}
\theoremstyle{definition}%
\newtheorem{remark}{Remark}%
\newtheorem{definition}{Definition}

\begin{document}

\title[Signs of Hamiltonian Circles in  Simple Plane Signed Graphs]{Signs of Hamiltonian Circles in  Simple Plane Signed Graphs}

\author{Xiyong Yan}
\subjclass[2020]{Primary: 05C45 (Hamiltonian graphs); Secondary: 05C22 (Signed and weighted graphs)}
\address{89 Park Ave Apt 29,  Binghamton, NY,  USA,  13903.}

\begin{abstract}
We study which signs can occur among Hamiltonian circles in  simple plane signed graphs.  Using a face-based viewpoint, we relate the sign of a Hamiltonian
circle to the product of the signs of the faces inside it, and we introduce co-Hamiltonian sequences.
This yields a criterion for the existence of opposite-sign Hamiltonian circles via two co-Hamiltonian sequences
with opposite face-products.  Motivated by signed grid graphs, we develop local structural theorems that allow one
to certify the existence of both signs without explicitly constructing the full sequences, including a ladder-type
configuration where toggling along two $4$-circles produces Hamiltonian circles of opposite sign,  as well as hexagon configurations that realize both signs.

\end{abstract}

\maketitle

\section{Introduction}

We investigate which signs can occur among Hamiltonian circles in simple plane signed graphs. Motivated by a question of Zaslavsky \cite{1}, we ask when a signed plane graph that contains a Hamiltonian circle must contain both a positive and a negative one.

Our approach is based on a face-oriented viewpoint. In a plane signed graph, the sign of a Hamiltonian circle can be expressed as the product of the signs of the bounded faces it encloses. This leads to the notion of a \emph{co-Hamiltonian sequence}, which describes a systematic way to remove faces while preserving 2-connectedness and ultimately leaving a unique Hamiltonian circle. Using this framework, we obtain a criterion for the existence of opposite-sign Hamiltonian circles in terms of two co-Hamiltonian sequences whose face-products differ.

Motivated by signed grid graphs, we further develop local structural results that allow one to certify the existence of both signs without constructing full sequences. In particular, ladder-type configurations and certain hexagon structures force the realization of both positive and negative Hamiltonian circles. These results show that the global sign behavior of Hamiltonian circles in plane signed graphs is governed by local face configurations.

\section{Hamiltonian Circle in  Simple Plane Signed Graphs}

We now study Hamiltonian circles in  simple plane signed graphs.  In contrast to the complete graph case, the
fixed embedding restricts how Hamiltonian circles can run through the graph, but it also provides additional
structure through faces.  Our approach is to express the sign of a Hamiltonian circle in terms of the signs of
the faces it encloses, and then to control these face products by removing faces from the outside inward.

To formalize this process, we introduce co-Hamiltonian sequences and the associated Hamiltonian sets, together
with the weak dual (face graph) as a bookkeeping tool.  This yields a criterion for the existence of Hamiltonian
circles of opposite sign in terms of two co-Hamiltonian sequences with opposite face products, and it motivates
later practical tests based on local configurations in signed grid graphs and ladder-type subgraphs.

\begin{definition}[Outer face, outer edges, exterior and interior vertices]
Let $\Sigma$ be a simple plane signed graph, and let $F_0$ denote its (unique) unbounded face, called the \emph{outer face}.
An edge $e$ of $\Sigma$ is an \emph{outer edge} if it is incident with $F_0$, equivalently, if $e$ lies on the boundary $\partial F_0$.
A vertex $v$ of $\Sigma$ is an \emph{exterior vertex} if it is incident with $F_0$, and an \emph{interior vertex} otherwise.
\end{definition}

\begin{definition}[Co-Hamiltonian edge sequence, co-Hamiltonian face sequence, and Hamiltonian set]\label{coh}
Let $G$ be a $2$-connected plane graph with outer face $F_0$, and let $\mathcal{B}$ be the set of bounded faces of $G$.

\smallskip
\noindent\textbf{(1) Co-Hamiltonian edge sequence.}
Let $G$ be a $2$-connected plane graph with outer face $F_0(G)$.
For an ordered edge sequence
\[
E'=(e_1,e_2,\dots,e_k),
\]
set
\[
G_0:=G
\quad\text{and}\quad
G_t:=G-\{e_1,\dots,e_t\}\ \ (t=1,\dots,k).
\]
We call $E'$ a \emph{co-Hamiltonian edge sequence} if for each $t=1,\dots,k$,
\begin{enumerate}
\item $e_t$ lies on the boundary of the outer face of $G_{t-1}$, and
\item $G_t$ is $2$-connected,
\end{enumerate}
and in the final graph $G_k$ every vertex is exterior (i.e., incident with the outer face) and
$G_k$ contains exactly one Hamiltonian circle.

\smallskip
\noindent\textbf{(2) Co-Hamiltonian face sequence.}
During the deletion process above, each deleted edge $e_t$ is required to lie on the boundary of the outer face of $G_{t-1}$.
Hence deleting $e_t$ merges the outer face of $G_{t-1}$ with a unique bounded face of $G_{t-1}$; denote that face by $F_t$.
The resulting ordered face sequence
\[
L=(F_1,F_2,\dots,F_k)
\]
is called the \emph{co-Hamiltonian (face) sequence} induced by $E'$.
(Equivalently, $F_t$ is the bounded face that becomes part of the outer face at step $t$.)

\smallskip
\noindent\textbf{(3) Hamiltonian set.}
Let $H$ be the set of bounded faces of the final graph $G_k$ (equivalently, the faces in $\mathcal{B}$ that are not merged into the outer face during the process).
Then $H$ is called a \emph{Hamiltonian set} of $G$, and its unique Hamiltonian circle is called the \emph{Hamiltonian circle determined by $H$}.
\end{definition}

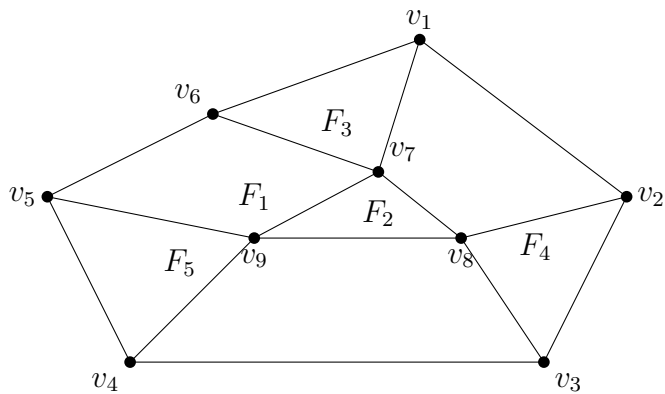
\begin{figure}[H]
\centering
\begin{tikzpicture}[scale=1.1, line join=round]

\coordinate (v1) at (0,0);
\coordinate (v2) at (5,0);
\coordinate (v3) at (6,2);
\coordinate (v4) at (3.5,3.9);
\coordinate (v5) at (1,3);
\coordinate (v6) at (-1,2);

\coordinate (u1) at (1.5,1.5);   
\coordinate (u2) at (4,1.5);     
\coordinate (u3) at (3,2.3);     

\draw (v1)--(v2)--(v3)--(v4)--(v5)--(v6)--cycle;

\draw (v6)--(u1)--(v1);
\draw (u1)--(u3)--(v5);
\draw (u3)--(u2)--(v3);
\draw (u2)--(v2);
\draw (u1)--(u2);
\draw (v4)--(u3);

\foreach \p in {v1,v2,v3,v4,v5,v6,u1,u2,u3}
  \fill (\p) circle (2pt);

\node[above]      at (v4) {$v_1$};
\node[right]      at (v3) {$v_2$};
\node[below right]at (v2) {$v_3$};
\node[below left] at (v1) {$v_4$};
\node[left]       at (v6) {$v_5$};
\node[above left] at (v5) {$v_6$};

\node[above right] at (u3) {$v_7$};
\node[below]       at (u2) {$v_8$};
\node[below]       at (u1) {$v_9$};

\node at (2.5,2.9) {$F_3$};
\node at (3.0,1.8) {$F_2$};
\node at (0.6,1.2) {$F_5$};
\node at (4.9,1.4) {$F_4$};
\node at (1.5,2) {$F_1$};

\end{tikzpicture}
\caption{A plane graph with three interior vertices}
\label{3b}
\end{figure}

\textbf{Example.}
In Figure~\ref{3b} the graph has three interior vertices $v_7,v_8,v_9$.
Consider the face sequence $L_1=(F_1,F_2)$.
First delete the outer edge $v_5v_6$.
This merges the bounded face $F_1$ with the outer face, so $F_1$ is no longer a bounded face and the vertices
$v_7$ and $v_9$ become exterior.
Next delete the edge $v_7v_9$; after the first deletion this edge lies on the boundary of the outer face, so its
deletion merges the bounded face $F_2$ with the outer face and makes $v_8$ exterior as well.
After these two deletions the remaining graph is still $2$-connected.
Moreover, the remaining bounded faces form a Hamiltonian set and the associated subgraph contains exactly one
Hamiltonian circle.
Thus $L_1$ is a co-Hamiltonian sequence (Definition~\ref{coh}).
Similarly, $L_2=(F_3,F_4,F_5)$ is another co-Hamiltonian sequence.

\begin{definition}[Weak dual or face graph]
Let $G$ be a plane graph with outer face $F_0$.
The \emph{weak dual} of $G$, also called the \emph{face graph}, and denoted by $D(G)$,
is the graph whose vertices correspond to the bounded faces of $G$.
Two vertices of $D(G)$ are adjacent if the corresponding bounded faces of $G$
share an edge.
\end{definition}

\begin{definition}[Outerplane graph]
A \emph{simple outerplane graph} is a simple plane graph in which every vertex lies on the boundary of the outer face $F_0$.

\end{definition}

\begin{definition}[Outer boundary circle]
Let $G$ be a $2$-connected plane graph with outer face $F_0$.
The boundary $\partial F_0$ is the closed walk formed by the edges incident with $F_0$.
In a $2$-connected plane graph, $\partial F_0$ is a simple circle, called the
\emph{outer boundary circle} of $G$.
\end{definition}

\begin{lemma}\label{lemmatree}
Let $G$ be a $2$-connected simple outerplane graph with outer face $F_0$, and assume that $G$ has at least one
bounded face. Then $D(G)$ is a tree.
\end{lemma}

\begin{proof}
We show that $D(G)$ is connected and has no circle.

\smallskip
\noindent\emph{(1) $D(G)$ is connected.}
Assume for a contradiction that $D(G)$ is disconnected, and let $\mathcal F_1,\mathcal F_2,...,\mathcal F_k$ be  nonempty
components of $D(G)$. If $k=1,$ then $D(G)$ is connected. Thus, we may assume $k>1$. For $i=1,2,...,k$ set
\[
G_i:=\bigcup_{f\in\mathcal F_i}\partial f.
\]
If a bounded face in $\mathcal F_1$ shared an edge with a bounded face in $\mathcal F_2$, then the corresponding
vertices of $D(G)$ would be adjacent, impossible. Hence $G_1$ and $G_2$ are edge disjoint.

Because $G$ is $2$-connected, it has no bridges, so every edge of $G$ lies on the boundary of some bounded face.
If $V(G_1)\cap V(G_2)=\emptyset$, then $G$ is disconnected,
contradiction. So $V(G_1)\cap V(G_2)\neq\emptyset$.

We claim that $|V(G_1)\cap V(G_2)|\le 1$.  Suppose $G_1$ and $G_2$ share two distinct vertices $v_1\neq v_2$.
Let $O_i$ be the boundary circle of the outer face of $G_i$ in the inherited embedding. Since $G$ is outerplane,
all vertices lie on $\partial F_0$, hence $O_1$ and $O_2$ lie on $\partial F_0$ as well. Choose $v_1,v_2$ so that
the $v_1$--$v_2$ arcs $P_1\subseteq O_1$ and $P_2\subseteq O_2$ satisfy
\[
V(P_1)\cap V(P_2)=\{v_1,v_2\}.
\]
Then $P_1\cup P_2$ is a simple circle in $G$.

If both $P_1$ and $P_2$ are the single edge $v_1v_2$, then $G$ has two parallel edges between $v_1$ and $v_2$,
contradicting that $G$ is simple (Figure~\ref{figoo}, left). Otherwise, at least one of $P_1,P_2$ has an interior
vertex; say $P_2$ contains $v_3\notin\{v_1,v_2\}$. Since $P_1\cup P_2$ is a  circle, the vertex $v_3$ lies
 inside it, hence $v_3$ is not incident with $F_0$, contradicting that $G$ is outerplane
(Figure~\ref{figoo}, middle). This proves $|V(G_1)\cap V(G_2)|\le 1$.

Thus $V(G_1)\cap V(G_2)=\{v\}$ for some vertex $v$. Suppose $k=2$.
Since $G_1$ and $G_2$ are edge-disjoint and meet only at $v$,
the graph $G-v$ is disconnected. Hence $v$ is a cut vertex of $G$,
contradicting $2$-connectedness. Suppose $k>2$. If $v$ is a vertex cut, then we get a contradiction.
If $v$ is not a vertex cut, then there exist
$G_{i_1},\dots,G_{i_r}\in\{G_1,\dots,G_k\}$
such that the sequence of subgraphs
$G_1,G_2,G_{i_1},\dots,G_{i_r},G_1$
has the property that each pair of consecutive subgraphs shares a common vertex.
However, these subgraphs enclose a bounded face $F$ (see Figure~\ref{figoo}, right),
and $F$ connects all $G_i$, a contradiction.

Therefore $D(G)$ is connected.

\smallskip
\noindent\emph{(2) $D(G)$ is acyclic.}
Suppose for a contradiction that $D(G)$ contains a circle
\[
f_1,f_2,\dots,f_k,f_1\qquad (k\ge 3),
\]
where consecutive faces share an edge. Let $e_i$ be the edge common to $f_i$ and $f_{i+1}$ (indices mod $k$).
Then the dual edges $e_1^*,\dots,e_k^*$ form a circle in the planar dual $G^*$, hence they trace a simple closed
curve in the plane. By the Jordan curve theorem, this curve bounds a region that is disjoint from the outer face.
In particular, that bounded region contains a vertex of $G$ that is not incident with $F_0$, contradicting that
$G$ is outerplane. Hence $D(G)$ has no circle, so it is acyclic.

\smallskip
Since $D(G)$ is connected and acyclic, it is a tree.
\end{proof}

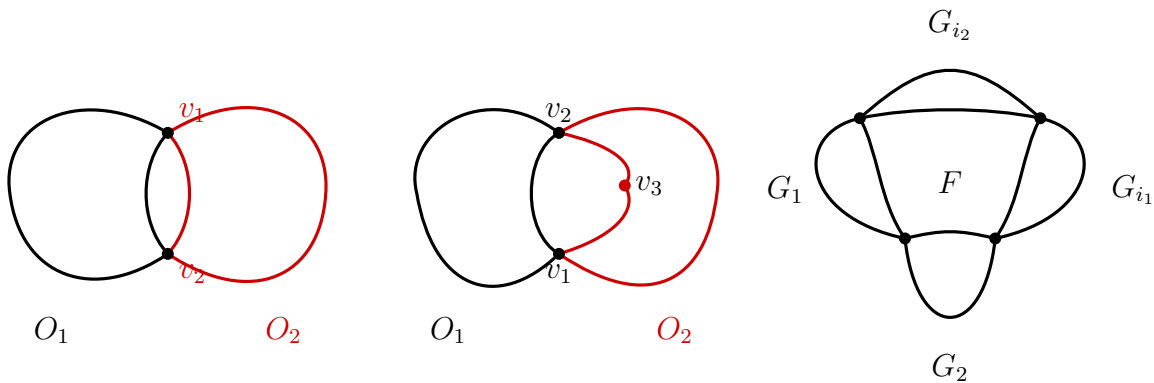
\begin{figure}[H]
\centering
\begin{tikzpicture}[scale=1, line cap=round, line join=round]
\tikzset{
  bnd/.style={line width=1.2pt},
  redbnd/.style={bnd,draw=red!80!black},
  dot/.style={circle,fill=black,inner sep=1.6pt},
  rdot/.style={circle,fill=red!80!black,inner sep=1.6pt}
}


\def\ThirdScalee{.7} 
\begin{scope}[shift={(0,0)},scale=\ThirdScalee]
  \coordinate (v1) at (0,1.15);
  \coordinate (v2) at (0,-1.15);

  \draw[bnd]
    (v1)
      .. controls (-1.6,2.1) and (-3.2,1.4) .. (-3.0,0.0)
      .. controls (-2.8,-1.6) and (-1.4,-2.1) .. (v2)
      .. controls (-0.55,-0.55) and (-0.55,0.55) .. (v1);

  \draw[redbnd]
    (v1)
      .. controls (1.7,2.2) and (3.1,1.4) .. (3.0,0.0)
      .. controls (2.9,-1.6) and (1.6,-2.2) .. (v2)
      .. controls (0.55,-0.55) and (0.55,0.55) .. (v1);

  \node[dot] at (v1) {};
  \node[dot] at (v2) {};
  \node[red!80!black, above right] at (v1) {$v_1$};
  \node[red!80!black, below right] at (v2) {$v_2$};

  \node at (-2.2,-2.6) {$O_1$};
  \node[red!80!black] at (2.2,-2.6) {$O_2$};
\end{scope}

\def\ThirdScaleee{.7} 
\begin{scope}[shift={(5.2,0)},scale=\ThirdScaleee]
  \coordinate (w2) at (0,1.15);
  \coordinate (w1) at (0,-1.15);
  \coordinate (w3) at (1.25,0.15);

  \draw[bnd]
    (w2)
      .. controls (-1.4,2.2) and (-3.0,1.2) .. (-2.7,0.0)
      .. controls (-2.4,-1.8) and (-1.2,-2.3) .. (w1)
      .. controls (-0.7,-0.75) and (-0.7,0.75) .. (w2);

  \draw[redbnd]
    (w2)
      .. controls (1.8,2.2) and (3.2,1.3) .. (3.0,0.0)
      .. controls (2.8,-1.7) and (1.7,-2.3) .. (w1);

  \draw[redbnd]
    (w2)
      .. controls (0.95,0.95) and (1.55,0.55) .. (w3)
      .. controls (1.55,-0.35) and (1.00,-0.85) .. (w1);

  \node[dot] at (w2) {};
  \node[dot] at (w1) {};
  \node[rdot] at (w3) {};

  \node[above] at (w2) {$v_2$};
  \node[below] at (w1) {$v_1$};
  \node[right] at (w3) {$v_3$};

  \node at (-2.1,-2.6) {$O_1$};
  \node[red!80!black] at (2.2,-2.6) {$O_2$};
\end{scope}


\def\ThirdScale{1.0} 

\begin{scope}[shift={(10.4,0)}, scale=\ThirdScale, line cap=round, line join=round]
  \tikzset{
    bnd/.style={line width=1.2pt},
    dot/.style={circle,fill=black,inner sep=1.6pt}
  }

  \coordinate (a) at (-1.2, 1.0);  
  \coordinate (b) at ( 1.2, 1.0);  
  \coordinate (c) at ( 0.6,-0.6);  
  \coordinate (d) at (-0.6,-0.6);  

  \draw[bnd]
    (a) .. controls (-0.5,1.15) and (0.5,1.15) .. (b);

  \draw[bnd]
    (a) .. controls (-0.95,0.55) and (-0.95,-0.10) .. (d);

  \draw[bnd]
    (b) .. controls (0.95,0.55) and (0.95,-0.10) .. (c);

  \draw[bnd]
    (d) .. controls (-0.10,-0.48) and (0.10,-0.48) .. (c);

  \node at (0,0.15) {$F$};


  \draw[bnd]
    (a) .. controls (-0.25,1.85) and (0.25,1.85) .. (b);

  \draw[bnd]
    (a) .. controls (-2.10,0.80) and (-2.00,-0.30) .. (d);

  \draw[bnd]
    (b) .. controls ( 2.10,0.80) and ( 2.00,-0.30) .. (c);

  \draw[bnd]
    (d) .. controls (-0.45,-2.00) and (0.45,-2.00) .. (c);

  \node[dot] at (a) {};
  \node[dot] at (b) {};
  \node[dot] at (c) {};
  \node[dot] at (d) {};

  \node[above] at (0,1.90) {$G_{i_2}$};
  \node[left]  at (-1.8,0.05) {$G_{1}$};
  \node[right] at ( 2,0.05) {$G_{i_1}$};
  \node[below] at (0,-2) {$G_2$};

\end{scope}

\end{tikzpicture}
\caption{If $O_1$ and $O_2$ share two vertices $v_1,v_2$, then either parallel edges occur (left) or an interior vertex $v_3$ is forced (middle). The rightmost sketch illustrates the decomposition into $G_1,G_2,G_{i_1},G_{i_2}$ around a face $F$.}
\label{figoo}
\end{figure}

\begin{lemma}\label{lem:weak-dual-tree-unique-ham}
Let $G$ be a $2$-connected simple outerplane graph with outer face $F_0$, and assume $G$ has at least one bounded face.
Then $G$ contains a unique Hamiltonian circle, namely the outer boundary circle $\partial F_0$.
\end{lemma}

\begin{proof}
Since $G$ is $2$-connected and outerplane, $\partial F_0$ is a simple circle containing every vertex of $G$,
so $\partial F_0$ is a Hamiltonian circle.

We prove uniqueness by induction on the number $b$ of bounded faces.

\smallskip
\noindent\emph{Base $b=1$.}
Then $G$ has exactly one bounded face $f$.
Every edge of $G$ lies on the boundary of a bounded face (no bridges in a $2$-connected plane graph), hence on $\partial f$.
Thus $G=\partial f=\partial F_0$, so there is exactly one Hamiltonian circle.

\smallskip
\noindent\emph{Inductive step.}
Assume $b\ge2$ and the statement holds for all such graphs with fewer than $b$ bounded faces.
By Lemma~\ref{lemmatree}, $D(G)$ is a tree, so it has a leaf face $f$.
Let $g=xy$ be the unique edge shared by $f$ and another bounded face; equivalently, $g$ is the unique edge of $\partial f$
not contained in $\partial F_0$.
Let $P:=\partial f - g$, the $x$--$y$ path along $\partial f$.

\medskip
\noindent\textbf{Claim.} Every Hamiltonian circle $C$ of $G$ contains all edges of $P$ and does not contain $g$.

\smallskip

\smallskip
Since $f$ is a leaf of $D(G)$, the edge $g$ is the \emph{only} edge of $\partial f$ that is not on $\partial F_0$.
Equivalently, every edge of $P$ lies on the outer boundary $\partial F_0$.

Let $u$ be an internal vertex of $P$ (so $u\neq x,y$). In an outerplane embedding, all edges incident with $u$
lie in the outer face except those belonging to faces that contain $u$. Because $f$ is the \emph{only} bounded face
on the side of $P$, there is no edge of $G$ that leaves $u$ to the interior of the disk bounded by $\partial f$.
Consequently, the only way for a Hamiltonian circle to pass through $u$ is to use the two boundary edges of $P$
incident with $u$. Thus every Hamiltonian circle must contain those two edges, and hence $P\subseteq C$.

Finally, if $g\in C$ as well, then $C$ contains every edge of $\partial f = P\cup\{g\}$.
Since $\partial f$ is a circle and $f$ is bounded, this would make $\partial f$ a proper subcircle of the Hamiltonian
circle $C$, which is impossible. Hence $g\notin C$.

\medskip
Now form $G_1$ by contracting the path $P$ to a single edge $xy$ embedded on the outer boundary.
Then $G_1$ is still simple, outerplane, and $2$-connected, and it has exactly $b-1$ bounded faces
(the leaf face $f$ is removed).

By induction, $G_1$ has a unique Hamiltonian circle, namely $\partial F_0(G_1)$.
By the Claim, every Hamiltonian circle of $G$ contains $P$ and avoids $g$, so contracting $P$ gives a Hamiltonian circle of $G_1$.
Conversely, expanding the edge $xy$ in $\partial F_0(G_1)$ back to the path $P$ yields a Hamiltonian circle of $G$.
Hence $G$ has exactly one Hamiltonian circle, and it is $\partial F_0$.
\end{proof}

\begin{lemma}\label{facegraph}
Let $\Sigma$ be a $2$-connected simple plane signed graph, let $H'$ be the set of all bounded faces of $\Sigma$,
and let $C_0$ be the boundary circle of the outer face.
Let $\sigma:E(G)\to\{+,-\}$ be an edge-signing.

For each bounded face $f\in H'$, define its face-sign by
\[
\sigma(f)=\prod_{g\in \partial f}\sigma(g).
\]
Then
\[
\sigma(C_0)=\prod_{f\in H'} \sigma(f).
\]
\end{lemma}

\begin{proof}
By definition,
\[
\sigma(C_0)=\prod_{g\in E(C_0)} \sigma(g).
\]

On the other hand,
\[
\prod_{f\in H'} \sigma(f)
=\prod_{f\in H'} \ \prod_{g\in \partial f} \sigma(g).
\]
Reordering the product, each edge $g$ of $G$ appears once for each bounded face incident with $g$.

We distinguish two cases.

\smallskip
\noindent\textbf{(i) Boundary edges.}
If an edge $g$ lies on the boundary circle $C_0$, then it is incident with exactly one bounded face.
Hence $\sigma(g)$ appears exactly once in the product.

\smallskip
\noindent\textbf{(ii) Interior edges.}
If an edge $g$ is shared by two bounded faces, then $\sigma(g)$ appears twice in the product.
Since $\sigma(g)\in\{+,-\}$, we have
\[
\sigma(g)^2=+.
\]
Thus all interior-edge contributions cancel.

\smallskip
Therefore, only boundary edges contribute nontrivially, and we obtain
\[
\prod_{f\in H'} \sigma(f)
=\prod_{g\in E(C_0)} \sigma(g)
=\sigma(C_0).
\]
\end{proof}

\begin{lemma}[Peeling Lemma]\label{lem:peeling}
Let $G$ be a $2$-connected plane graph that contains a Hamiltonian circle $C$.
Assume that $C$ is \emph{not} the boundary circle of the outer face of $G$
(equivalently, some edge of $G$ on the boundary of the outer face is not in $E(C)$).
Then there exists an edge $e\notin E(C)$ on the boundary of the outer face of $G$
such that $G-e$ is still $2$-connected.
\end{lemma}

\begin{proof}
Since $C$ is not the boundary circle of the outer face, the boundary walk of the outer face uses at least one edge
that is not in $C$. Choose such an edge and call it $e$. Then $e$ lies on the boundary of the outer face and
$e\notin E(C)$.

Now consider $G-e$. The circle $C$ is still present as a Hamiltonian circle in $G-e$ (because we did not delete any edge of $C$).
A graph that contains a Hamiltonian circle is $2$-connected: for any two distinct vertices $x,y$ on $C$,
the circle $C$ provides two internally vertex-disjoint $x$--$y$ paths (the two arcs of $C$ between $x$ and $y$).
Hence $G-e$ is $2$-connected.
\end{proof}

\begin{lemma}\label{thm:ham-iff-coham-seq}
Let $G$ be a $2$-connected plane graph with outer face $F_0$, and let $\mathcal{B}$ be the set of bounded faces of $G$.
Then $G$ contains a Hamiltonian circle if and only if $G$ admits a co-Hamiltonian (face) sequence
in the sense of Definition~\ref{coh}.
\end{lemma}

\begin{proof}

($\Rightarrow$)

Assume $G$ contains a Hamiltonian circle $C$, and let $H$ be the set of bounded faces inside the closed disk bounded by $C$.
Let $\mathcal{B}_{\mathrm{out}}:=\mathcal{B}\setminus H$ denote the bounded faces outside $C$.
If $\mathcal{B}_{\mathrm{out}}=\emptyset$, then $C=\partial F_0$ and the empty sequence is a co-Hamiltonian face sequence.
Otherwise, we iteratively delete edges to merge the faces of $\mathcal{B}_{\mathrm{out}}$ into the outer face,
while preserving $2$-connectedness and never deleting an edge of $C$.

At each step, the current graph $G_{t-1}$ is $2$-connected and still contains the Hamiltonian circle $C$
(as a subgraph), and $G_{t-1}\neq C$ because there remains at least one edge outside $C$.
Therefore, by the Peeling Lemma (Lemma~\ref{lem:peeling}), there exists an edge
\(e_t \notin E(C)\) that lies on the boundary of the outer face of \(G_{t-1}\)
such that deleting \(e_t\) preserves \(2\)-connectedness.
Define
\[
G_t:=G_{t-1}-e_t .
\]
Let $F_t$ be the bounded face of $G_{t-1}$ that is merged into the outer face when $e_t$ is deleted. Repeating this process  produces a co-Hamiltonian edge sequence
\((e_1,\dots,e_k)\) and the corresponding co-Hamiltonian face sequence
\((F_1,\dots,F_k)\). The process terminates when all faces in $\mathcal{B}_{\mathrm{out}}$ have been merged into the outer face, i.e. when the bounded faces of the final graph are exactly $H$.
By construction, every vertex of the final graph is exterior and it contains exactly one Hamiltonian circle (namely $C$), so $H$ is the Hamiltonian set determined by the sequence.
Hence $G$ admits a co-Hamiltonian (face) sequence.

($\Leftarrow$)
Conversely, assume $G$ admits a co-Hamiltonian edge sequence $E'=(e_1,\dots,e_k)$, with induced face sequence
$L=(F_1,\dots,F_k)$ and final graph $G_k$ as in Definition~\ref{coh}.
By definition, $G_k$ contains exactly one Hamiltonian circle. In particular, $G_k$ contains a Hamiltonian circle.
Since $G_k$ is a spanning subgraph of $G$ (obtained by deleting edges only), that Hamiltonian circle is also a
Hamiltonian circle of $G$.
Therefore $G$ contains a Hamiltonian circle.
\end{proof}

\subsection{Grid graphs}
Let $\Sigma$ be an $m \times n$ signed grid. We label the vertices by
\[
V = \{(i,j) \mid i \in \{1,2,\dots,m\},\; j \in \{1,2,\dots,n\}\}.
\] 
The grid contains $(m-1)(n-1)$ unit boxes (or faces). 
Let $B$ denote the set of all boxes. 
We label the boxes by
\[
B = \{[i,j] \mid i \in \{1,2,\dots,m-1\},\; j \in \{1,2,\dots,n-1\}\}.
\]

We label the edges of the grid as follows.  
For each vertex $(i,j)$ with $1 \le i \le m$ and $1 \le j \le n$:

\begin{enumerate}

\item the \emph{horizontal edge} joining $(i,j)$ and $(i,j+1)$,
for $1 \le i \le m$ and $1 \le j \le n-1$,
is denoted by
\[
e^h_{i,j}=\big((i,j),(i,j+1)\big);
\]

\item the \emph{vertical edge} joining $(i,j)$ and $(i+1,j)$,
for $1 \le i \le m-1$ and $1 \le j \le n$,
is denoted by
\[
e^v_{i,j}=\big((i,j),(i+1,j)\big).
\]
\end{enumerate}

We present two examples: one that contains a Hamiltonian circle and one that does not.

In Figure~\ref{3by4}, there are exactly two interior vertices, namely $(2,2)$ and $(2,3)$. 
All other vertices lie on the outer face.
If we remove the vertical edge $e^{v}_{1,2}$, equivalently remove the box $[1,2]$, 
then it is straightforward to verify that the resulting graph contains exactly one Hamiltonian circle.
Moreover, in the resulting graph the vertices $(2,2)$ and $(2,3)$ become exterior vertices.

This illustrates a general phenomenon.
Removing an edge incident with the outer face converts two interior vertices into exterior vertices
in the resulting graph.

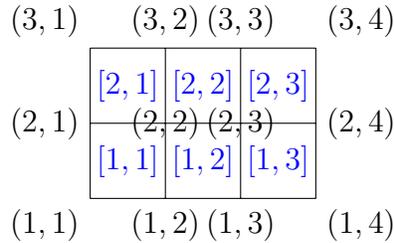
\begin{figure}[H]
\centering
\begin{tikzpicture}[scale=1]
  \draw (0,0) grid (3,2);

  \node[below left]  at (0,0) {$(1,1)$};
  \node[below]       at (1,0) {$(1,2)$};
  \node[below]       at (2,0) {$(1,3)$};
  \node[below right] at (3,0) {$(1,4)$};

  \node[left]  at (0,1) {$(2,1)$};
  \node        at (1,1) {$(2,2)$};
  \node        at (2,1) {$(2,3)$};
  \node[right] at (3,1) {$(2,4)$};

  \node[above left]  at (0,2) {$(3,1)$};
  \node[above]       at (1,2) {$(3,2)$};
  \node[above]       at (2,2) {$(3,3)$};
  \node[above right] at (3,2) {$(3,4)$};

  \node[blue] at (0.5,0.5) {$[1,1]$};
  \node[blue] at (0.5,1.5) {$[2,1]$};

  \node[blue] at (1.5,0.5) {$[1,2]$};
  \node[blue] at (1.5,1.5) {$[2,2]$};

  \node[blue] at (2.5,0.5) {$[1,3]$};
  \node[blue] at (2.5,1.5) {$[2,3]$};
\end{tikzpicture}
\caption{3 by 4 grid with face (box) labels in blue and vertex labels in parentheses}
\label{3by4}
\end{figure}

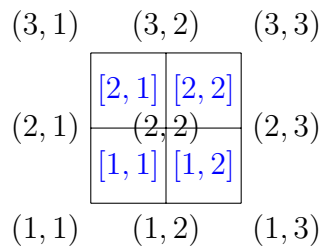
\begin{figure}[H]
\centering
\begin{tikzpicture}[scale=1]
  \draw (0,0) grid (2,2);

  \node[below left]  at (0,0) {$(1,1)$};
  \node[below]       at (1,0) {$(1,2)$};
  \node[below right] at (2,0) {$(1,3)$};

  \node[left]  at (0,1) {$(2,1)$};
  \node        at (1,1) {$(2,2)$};
  \node[right] at (2,1) {$(2,3)$};

  \node[above left]  at (0,2) {$(3,1)$};
  \node[above]       at (1,2) {$(3,2)$};
  \node[above right] at (2,2) {$(3,3)$};

  \node[blue] at (0.5,0.5) {$[1,1]$};
  \node[blue] at (1.5,0.5) {$[1,2]$};
  \node[blue] at (0.5,1.5) {$[2,1]$};
  \node[blue] at (1.5,1.5) {$[2,2]$};

\end{tikzpicture}
\caption{3 by 3 grid with labels $(x,y)$ replaced by $(y,x)$}
\label{3by3}
\end{figure}

Note that not every grid graph contains a Hamiltonian circle.
For example, consider the graph in Figure~\ref{3by3}.
If we remove any exterior edge from the graph, say the vertical edge $e^{v}_{1,2}$, then the vertex $(1,3)$ has degree~$1$.
Consequently, the remaining graph cannot contain a Hamiltonian circle.

Alternatively, observe that removing an exterior edge from a grid graph turns two interior vertices into exterior vertices.
However, the graph in Figure~\ref{3by3} has only one interior vertex.
Thus, after the removal of any exterior edge, the resulting graph has no interior vertex structure compatible with a Hamiltonian circle.

The combinatorial structure of rectangular grid graphs imposes strong parity constraints. In particular, the parity of the interior vertices yields an immediate obstruction to Hamiltonicity, as stated in the following lemma.

\begin{lemma}
A rectangular $m\times n$ grid graph with an odd number of interior vertices cannot contain a Hamiltonian circle.
\end{lemma}

\begin{proof}
In the $m\times n$ grid graph, the interior vertices are exactly those $(i,j)$ with
$2\le i \le m-1$ and $2\le j \le n-1$, hence the number of interior vertices is
\[
(m-2)(n-2).
\]
If $(m-2)(n-2)$ is odd, then both $m-2$ and $n-2$ are odd, so $m$ and $n$ are odd. Therefore the total number of vertices is
\[
|V|=mn,
\]
which is odd.

Now the grid graph is bipartite under the coloring given by the parity of $i+j$:
\[
(i,j)\in X \iff i+j \text{ is even}, 
\qquad
(i,j)\in Y \iff i+j \text{ is odd}.
\]
Every edge joins a vertex in $X$ to a vertex in $Y$, so this is a bipartition.

In any circle of a bipartite graph, vertices alternate between the two parts, so every circle uses the same number of vertices from $X$ and from $Y$. In particular, a Hamiltonian circle would use all vertices, forcing
\[
|X|=|Y|.
\]
But if $|V|=|X|+|Y|$ is odd, then $|X|\neq |Y|$, a contradiction. Hence the grid graph cannot contain a Hamiltonian circle.
\end{proof}

We illustrate the next lemma (Lemma~\ref{lem:coham}) with a $4\times6$ grid graph: after deleting the sequence $L$, the remaining graph is shown in Figure~\ref{4by6}.

\begin{figure}[H]
\centering
\begin{tikzpicture}[scale=1, line cap=round, line join=round]

\foreach \x in {0,...,5}{
  \foreach \y in {0,...,3}{
    \coordinate (v\x\y) at (\x,\y);
  }
}


\foreach \x in {0,...,4}{
  \draw (v\x2)--(v\the\numexpr\x+1\relax2);
  \draw (v\x3)--(v\the\numexpr\x+1\relax3);
}
\foreach \x in {0,...,5}{
  \draw (v\x2)--(v\x3);
}

\foreach \x in {0,1,2,3,4,5}{
  \draw (v\x2)--(v\x1)--(v\x0);
}

\draw (v01)--(v11);  \draw (v00)--(v10);
\draw (v21)--(v31);  \draw (v20)--(v30);
\draw (v41)--(v51);  \draw (v40)--(v50);

\foreach \x in {0,...,5}{
  \foreach \y in {0,...,3}{
    \fill (v\x\y) circle (1.6pt);
  }
}

\node[left]       at (v03) {$(4,1)$};
\node[right]      at (v53) {$(4,6)$};
\node[left]       at (v00) {$(1,1)$};
\node[right]      at (v50) {$(1,6)$};

\node[left] at (1,0.5) {$[1,1]$};
\node[left] at (5,2.5) {$[3,5]$};

\end{tikzpicture}
\caption{The resulting graph obtained from a $4\times6$ grid after deleting a co-Hamiltonian sequence.}
\label{4by6}
\end{figure}
\begin{lemma}\label{lem:coham}
Let $\Sigma$ be an $m\times n$ rectangular grid with $n$ even.
There exists a co-Hamiltonian sequence $L$ consisting of
\[
\Big(\frac n2-1\Big)(m-2)
\]
boxes removed from the outside of $\Sigma$ such that the remaining graph contains exactly one Hamiltonian circle.
\end{lemma}

\begin{proof}
We give an explicit choice of $L$ and prove that it forces a unique Hamiltonian circle.

\medskip
\noindent\textbf{Construction of $L$.}
View $\Sigma$ as an $(m-1)\times (n-1)$ array of unit boxes labeled
\[
[i,j]\qquad (1\le i\le m-1,\ 1\le j\le n-1).
\]
Since $n$ is even, the set $\{2,4,6,\dots,n-2\}$ has size $\frac n2-1$.

For each row $i\in\{1,2,\dots,m-2\}$, remove the boxes
\[
[i,2],\ [i,4],\ [i,6],\ \dots,\ [i,n-2] \text{ in this order}.
\]
Thus we remove exactly $\frac n2-1$ boxes in each of the $m-2$ rows, so
\[
|L|=\Big(\frac n2-1\Big)(m-2).
\]

\medskip
Observe that the resulting graph $G_L$ is $2$-connected and that every vertex of $G_L$ is an exterior vertex.
Hence $G_L$ satisfies the hypotheses of Lemma~\ref{lem:weak-dual-tree-unique-ham}.
Therefore $G_L$ contains a \emph{unique} Hamiltonian circle, namely the boundary circle of its outer face.
In particular, the remaining graph contains exactly one Hamiltonian circle, as desired.
\end{proof}

Notation. For two sets of boxes $H_1'$ and $H_2'$, the symmetric difference
$H_1' \triangle H_2'$ denotes the set of boxes that are contained in $H_1'$ or $H_2'$, but not in both.

\begin{lemma}\label{lemma123} 
Let $\Sigma$ be an $m\times n$ grid with $m,n>3$ and $n$ even.
Let $H_1$ and $H_2$ be two distinct Hamiltonian circles with corresponding Hamiltonian sets
$H_1'$ and $H_2'$, respectively. Suppose
\[
H_1' \triangle H_2'=\{a,b\},
\]
where $a\in H_1'\setminus H_2'$ and $b\in H_2'\setminus H_1'$ are boxes. Then
\[
\sigma(H_1)=\sigma(H_2)\quad \Longleftrightarrow\quad \sigma(a)=\sigma(b).
\]
\end{lemma}

\begin{proof}
Let $S=H_1'\cap H_2'$. Since $H_1' \triangle H_2'=\{a,b\}$, we have
\[
H_1' = S\cup\{a\}, \qquad H_2' = S\cup\{b\}.
\]
Write $S=\{c_1,c_2,\dots,c_k\}$.  
Since $\sigma$ is multiplicative, by Lemma \ref{facegraph} we obtain
\[
\sigma(H_1)=\sigma(H_1')=\Big(\prod_{i=1}^k \sigma(c_i)\Big)\sigma(a),
\qquad
\sigma(H_2)=\sigma(H_2')=\Big(\prod_{i=1}^k \sigma(c_i)\Big)\sigma(b).
\]
Therefore,
\[
\sigma(H_1)=\sigma(H_2)
\iff
\Big(\prod_{i=1}^k \sigma(c_i)\Big)\sigma(a)
=
\Big(\prod_{i=1}^k \sigma(c_i)\Big)\sigma(b)
\iff
\sigma(a)=\sigma(b),
\]
as claimed.
\end{proof}

\begin{theorem}\label{thm:all-same-sign}
Let $\Sigma$ be an $m\times n$ grid with $m$ even and $m,n>3$.
 All Hamiltonian circles in $\Sigma$ have the same sign if and only if all boxes of $\Sigma$ except the four corner boxes have the same sign.
\end{theorem}

\begin{proof}
Suppose all Hamiltonian circles in $\Sigma$  have the same sign.

We will show  that all boxes of $\Sigma$, except the four corner boxes, have the same sign.
We prove this direction in two cases.

\medskip
\noindent\emph{Case 1: $n$ is even.}
Let $\mathcal{B}$ denote the set of all bounded faces of the graph $\Sigma$.

Let
\[
L:=\{[i,j]\mid i\in \{1,2,\dots,m-2\},\ j\in \{2,4,\dots,n-2\}\}
\]
be a co-Hamiltonian set of $\mathcal{B}$, and let $H=\mathcal{B}\setminus L$.
Fix a column index  $j\in\{2,4,\dots,n-2\}$, and each $i\in\{1,2,\dots,m-2\}$,
we may replace the box $[m-1,j]$ in $H$ by the box $[i,j]$
(that is, delete $[m-1,j]$ and add $[i,j]$).
Denote the resulting Hamiltonian set by $H_{i,j}$.
 Since the Hamiltonian circle in $H_{i,j}$
and the Hamiltonian circle in $H$ have the same signs and $H_{i,j}\triangle H=\{[i,j],[m-1,j]\}$ , by Lemma \ref{lemma123}, $\sigma([i,j])=\sigma([m-1,j])$. Since $i$ is arbitrary, all the boxes in the column $j$ have the same sign.

Consider $H$ again. Delete the box $[2,1]$ and add the box $[1,2]$
(equivalently, move the box $[2,1]$ to position $[1,2]$).
Denote the resulting Hamiltonian set by $H_2$.
Then, in $H_2$, for any $k\in\{3,4,\dots,m-2\}$, we may replace the box $[k,1]$
by the box $[2,1]$ (that is, move $[k,1]$ to position $[2,1]$).
Denote the resulting Hamiltonian set by $H_{k,1}$. Since all the Hamiltonian circles have the same sign and $H_{k,1}\triangle H_2=\{[2,1],[k,1]\}$, by Lemma \ref{lemma123}, $\sigma([k,1])=\sigma([2,1])$. Hence, all non-corner boxes in the first column must have the same sign. Similarly, all non-corner boxes in the last column must have the same sign.

Next, consider $\Sigma$ again. Let
\[
L_2:= \{[i,j]\mid i\in \{1,2,\dots,m-2\},\ j\in \{3,5,\dots,n-3\}\}
\;\cup\;
\{[i,1],[i,n-1]\mid i\in \{2,4,\dots,m-2\}\}.
\]
Let $H'=\mathcal{B}\setminus L_2$.
Then $H'$ contains exactly one Hamiltonian circle.

\medskip
\noindent\emph{Row movements in $H'$.}
Fix an odd column index $j\in\{3,5,\dots,n-3\}$.
For any $i\in\{1,2,\dots,m-2\}$, replace the box $[m-1,j]$ in $H'$
by the box $[i,j]$, and denote the resulting Hamiltonian set by $H_{i,j}$. Since we assume that all Hamiltonian circles in $\Sigma$ have the same sign, by Lemma~\ref{lemma123} we have
$\sigma([i,j])=\sigma([m-1,j])$.
Since $i$ is arbitrary, we conclude that all boxes in column $j$ have the same sign.

\medskip
\noindent\emph{Relations between different columns.}
Finally, we compare boxes lying in different columns.
In $H$, fix a number $j\in\{1,3,\dots,n-3\}$, we may move the box $[2,j]$
to position $[1,j+1]$. Denote the resulting Hamiltonian set by $H_{1,j+1}$. Since all the Hamiltonian circles have the same sign, 
by Lemma \ref{lemma123}, $\sigma([2,j])=\sigma([1,j+1])$. This implies the non-corner box sign in column $l$ is the same as the non-corner box sign in column $l+1$, for $l$ is odd.

Similarly, in $H$ again,  for each $j\in\{3,5,\dots,n-1\}$, we may move the box $[2,j]$
to position $[1,j-1]$. Denote the resulting Hamiltonian set by $H_{1,j-1}$.
With the same reasoning, $\sigma([2,j])=\sigma([1,j-1])$.  This implies the non-corner box sign in column $l$ is the same as the non-corner box sign in column $l+1$, for $l$ is even.
Hence, any two non-corner box in different column will have the same sign.

\medskip
\noindent\emph{Case 2: $n$ is odd.}

Let
\[
L:=\{[i,j]\mid i\in \{1,2,\dots,m-2\},\ j\in \{2,4,\dots,n-3\}\}
\;\cup\;
\{[i,n-1]\mid i\in \{2,4,\dots,m-2\}\}
\]
be a co-Hamiltonian set of $\Sigma$, and let $H=\mathcal{B}\setminus L$.
Then $H$ contains exactly one Hamiltonian circle.

Using the same row and column movement arguments as in Case~1,
we obtain that all non-corner boxes in the first $n-2$ columns
must have the same sign since all Hamiltonian circles in $\Sigma$
to have the same sign.

Next, consider a different co-Hamiltonian set
\[
L':=\{[i,j]\mid i\in \{1,2,\dots,m-2\},\ j\in \{3,5,\dots,n-2\}\}
\;\cup\;
\{[i,1]\mid i\in \{2,4,\dots,m-2\}\}.
\]
Using the same row and column movement arguments as in Case~1,
we conclude that all non-corner boxes in columns $2,3,\dots,n-1$
must have the same sign, since all Hamiltonian circles in $\Sigma$
are assumed to have the same sign.

Hence, from both cases we conclude that all boxes of $\Sigma$, except the four corner boxes, must have the same sign.

Conversely, assume that all boxes of $\Sigma$, except the four corner boxes, have the same sign.
We show that all Hamiltonian circles in $\Sigma$ have the same sign.

From the forward direction, we make the following two observations:
\begin{enumerate}
\item every Hamiltonian set contains all four corner boxes;
\item all Hamiltonian sets contain the same number of boxes.
\end{enumerate}

Let $H'$ be any Hamiltonian set, and let $C$ denote its corresponding Hamiltonian circle.
By Lemma \ref{facegraph},
\[
\sigma(C)=\prod_{[i,j]\in H'} \sigma([i,j]).
\]
Write
\[
H' = \{\text{four corner boxes}\} \cup \{\text{non-corner boxes in } H'\}.
\]

Since all non-corner boxes have the same sign, say $\alpha$, and since every Hamiltonian set
contains the same number of non-corner boxes, the contribution of the non-corner boxes to
$\sigma(C)$ is the same for every Hamiltonian circle.
Moreover, the contribution of the four corner boxes is fixed, since every Hamiltonian set
contains all four of them.

Therefore, for any two Hamiltonian circles $C_1$ and $C_2$, we have
\[
\sigma(C_1)=\sigma(C_2).
\]
Hence, all Hamiltonian circles in $\Sigma$ have the same sign.

\end{proof}

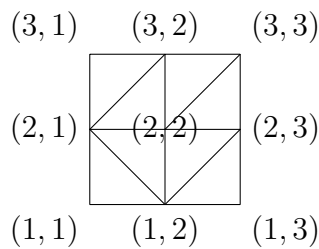
\begin{figure}[H]
\centering
\begin{tikzpicture}[scale=1]
  \draw (0,0) grid (2,2);

  \draw (0,1) -- (1,0);
  \draw (1,0) -- (2,1);
  \draw (0,1) -- (1,2);
  \draw (1,1) -- (2,2);

  \node[below left]  at (0,0) {$(1,1)$};
  \node[below]       at (1,0) {$(1,2)$};
  \node[below right] at (2,0) {$(1,3)$};

  \node[left]  at (0,1) {$(2,1)$};
  \node        at (1,1) {$(2,2)$};
  \node[right] at (2,1) {$(2,3)$};

  \node[above left]  at (0,2) {$(3,1)$};
  \node[above]       at (1,2) {$(3,2)$};
  \node[above right] at (2,2) {$(3,3)$};

\end{tikzpicture}
\caption{A $3\times 3$ grid with diagonals added to each square }
\label{3by3triangle3-swapped}
\end{figure}

\subsection{Triangulated grid graphs}

The $3\times 3$ grid graph does not contain a Hamiltonian circle (Figure~\ref{3by3}).
If we add a diagonal to each unit square, we obtain the triangulated grid shown in
Figure~\ref{3by3triangle3-swapped}, whose bounded faces are triangles.

Let $G_\triangle$ denote this triangulated grid.  Deleting a boundary edge of $G_\triangle$
merges one boundary triangle into the outer face.  For example, if we delete the horizontal edge
$e^h_{3,2}$ (as in Figure~\ref{3by3triangle3-swapped}), then the resulting graph contains a unique
Hamiltonian circle $C$.

In Figure~\ref{face22} below, deleting the vertical edge $e^v_{2,3}$, equivalently deleting the negative face on the right hand side, leaves a graph that is still $2$-connected, and it contains a unique Hamiltonian circle.
In the corresponding face graph, the vertex labeled $-2$ on the right is removed.
Consequently, the resulting face graph is a tree.

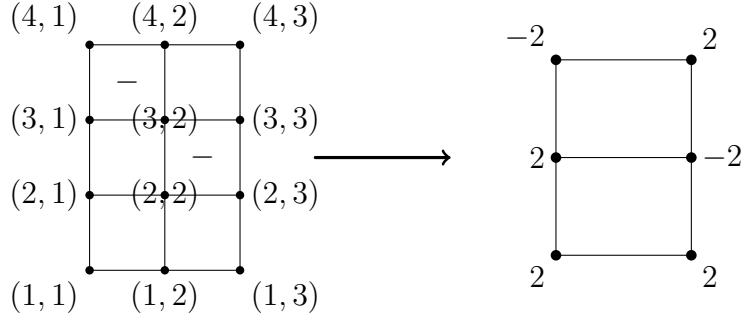
\begin{figure}[H]
\centering
\begin{tikzpicture}[line cap=round, line join=round, scale=1]

\begin{scope}[shift={(0,0)}]
  \draw (0,0) rectangle (2,3);
  \draw (1,0)--(1,3);
  \draw (0,1)--(2,1);
  \draw (0,2)--(2,2);

  \foreach \x in {0,1,2}{
    \foreach \y in {0,1,2,3}{
      \fill (\x,\y) circle (1.6pt);
    }
  }

 \node[below left]  at (0,0) {$(1,1)$};
\node[below]       at (1,0) {$(1,2)$};
\node[below right] at (2,0) {$(1,3)$};

\node[left]  at (0,1) {$(2,1)$};
\node        at (1,1) {$(2,2)$};
\node[right] at (2,1) {$(2,3)$};

\node[left]  at (0,2) {$(3,1)$};
\node        at (1,2) {$(3,2)$};
\node[right] at (2,2) {$(3,3)$};

\node[above left]  at (0,3) {$(4,1)$};
\node[above]       at (1,3) {$(4,2)$};
\node[above right] at (2,3) {$(4,3)$};

  \node at (0.5,2.5) {$-$};
  \node at (1.5,1.5) {$-$};
\end{scope}

\draw[->, line width=1.2pt] (3.0,1.5) -- (4.8,1.5);

\begin{scope}[shift={(6.2,0.2)}]
  \coordinate (a) at (0,2.6);
  \coordinate (b) at (1.8,2.6);
  \coordinate (c) at (0,1.3);
  \coordinate (d) at (1.8,1.3);
  \coordinate (e) at (0,0);
  \coordinate (f) at (1.8,0);

  \draw (a)--(b);
  \draw (c)--(d);
  \draw (e)--(f);
  \draw (a)--(c)--(e);
  \draw (b)--(d)--(f);

  \foreach \P in {a,b,c,d,e,f}{
    \fill (\P) circle (2pt);
  }

  \node[above left]  at (a) {$-2$};
  \node[above right] at (b) {$2$};
  \node[left]        at (c) {$2$};
  \node[right]       at (d) {$-2$};
  \node[below left]  at (e) {$2$};
  \node[below right] at (f) {$2$};
\end{scope}

\end{tikzpicture}
\caption{A $4\times3$ grid with two negative and four positive faces, and its face graph.}
\label{face22}
\end{figure}

\begin{definition}[Triangulation and face map]
Let $\Sigma$ be a $2$-connected plane signed graph with outer face $F_0$, and let
$\mathcal{B}$ denote the set of bounded faces of $\Sigma$.

For a bounded face $f\in\mathcal{B}$, a \emph{triangulation of $f$}
is a subdivision of $f$ into triangles by adding noncrossing diagonals
between vertices on the boundary of $f$.
Let $\tau(f)$ denote the number of triangles in any such triangulation of $f$.
Since $f$ is a polygon, the value $\tau(f)$ is independent of the chosen triangulation.

Define the \emph{face map}
\[
\phi:\mathcal{B}\to \mathbb{Z}
\]
by
\[
\phi(f)=\sigma(f)\,\tau(f),
\]
where $\sigma(f)\in\{\pm1\}$ is the sign of the face $f$.
\end{definition}

\begin{definition}[Removable face-vertex]
Let $D(\Sigma)$ be the face graph, and let $v_f$ be the vertex corresponding to a bounded face $f$.
We call $v_f$ \emph{removable} if
\begin{equation}\label{eq:removable}
\deg_{D(\Sigma)}(v_f)=|\phi(f)|+1.
\end{equation}
\end{definition}

To help a computer search for a Hamiltonian set, we can work in the face graph $D(\Sigma)$.
Label each vertex $v_f$ of $D(\Sigma)$ by the pair
\[
\bigl(\phi(f),\,\deg_{D(\Sigma)}(v_f)\bigr),
\]
where $f$ is the corresponding bounded face of $\Sigma$.
Using \eqref{eq:removable}, we may delete removable vertices one at a time.
The goal is to continue this elimination process until the remaining face graph is a tree (or reduces to a
single vertex), at which point the surviving faces form a candidate Hamiltonian set.

\begin{remark}
After deleting a vertex from the face graph, we obtain a new face graph (the induced subgraph on the remaining
face-vertices). In particular, the degrees of the remaining vertices may change after each deletion.
\end{remark}

\subsection{Criteria and local constructions for opposite-sign Hamiltonian circles}
 In Theorem~\ref{thm111}, our goal is to find two co-Hamiltonian sequences with distinct signs, which in turn
yield two Hamiltonian circles of opposite sign.  In practice, however, it may not be necessary to explicitly
construct such sequences.

Motivated by the behavior of signed complete graphs and signed grid graphs, we obtain the following theorems.
Under certain local face configurations, one can quickly determine the existence of both a positive Hamiltonian
circle and a negative Hamiltonian circle.

\begin{theorem}\label{thm111}
Let $\Sigma=(G,\sigma)$ be a simple plane signed graph that contains a Hamiltonian circle $C$.
Then $\Sigma$ contains both a positive Hamiltonian circle and a negative Hamiltonian circle
if and only if there exist two co-Hamiltonian sequences of faces
\[
\mathcal{F}_1 = (F_1,\dots,F_k)
\quad\text{and}\quad
\mathcal{F}_2 = (F'_1,\dots,F'_\ell),
\]
such that
\[
\prod_{i=1}^{k} \sigma(F_i)=-\prod_{j=1}^{\ell} \sigma(F'_j).
\]
\end{theorem}

\begin{proof}
Let $\mathcal{B}$ denote the set of all bounded faces of $G$.
For any (finite) set $S$ of bounded faces, write
\[
\sigma(S):=\prod_{f\in S}\sigma(f).
\]

\medskip
\noindent\emph{($\Leftarrow$).}
Assume there exist two co-Hamiltonian sequences
\[
\mathcal{F}_1=(F_1,\dots,F_k)
\qquad\text{and}\qquad
\mathcal{F}_2=(F'_1,\dots,F'_\ell)
\]
such that
\[
\prod_{i=1}^{k} \sigma(F_i)=-\prod_{j=1}^{\ell} \sigma(F'_j).
\]
Let $L_1:=\{F_1,\dots,F_k\}$ and $L_2:=\{F'_1,\dots,F'_\ell\}$, and define
\[
H_1:=\mathcal{B}\setminus L_1,
\qquad
H_2:=\mathcal{B}\setminus L_2.
\]
By the definition of a co-Hamiltonian sequence, $H_1$ determines a Hamiltonian circle $C_1$ and
$H_2$ determines a Hamiltonian circle $C_2$.
By Lemma~\ref{facegraph},
\[
\sigma(C_1)=\sigma(H_1),
\qquad
\sigma(C_2)=\sigma(H_2).
\]

Since 
\[
\sigma(\mathcal{B})=\sigma(H_1)\sigma(L_1)=\sigma(H_2)\sigma(L_2),
\]
 we have
\[
\sigma(H_1)=\sigma(\mathcal{B})\sigma(L_1),
\qquad
\sigma(H_2)=\sigma(\mathcal{B})\sigma(L_2).
\]
Therefore,
\[
\sigma(C_1)\sigma(C_2)
=\sigma(H_1)\sigma(H_2)
=\bigl(\sigma(\mathcal{B})\sigma(L_1)\bigr)\bigl(\sigma(\mathcal{B})\sigma(L_2)\bigr)
=\sigma(L_1)\sigma(L_2),
\]
because $\sigma(\mathcal{B})^2=1$.
The hypothesis $\sigma(L_1)=-\sigma(L_2)$ implies $\sigma(L_1)\sigma(L_2)=-1$, hence
\[
\sigma(C_1)\sigma(C_2)=-1,
\]
so $C_1$ and $C_2$ have opposite signs.

\medskip
\noindent\emph{($\Rightarrow$).}
Conversely, assume $G$ contains a positive Hamiltonian circle $C_+$ and a negative Hamiltonian circle $C_-$.
Let $H_+$ and $H_-$ be the sets of bounded faces inside $C_+$ and $C_-$, respectively, and set
\[
L_+:=\mathcal{B}\setminus H_+,
\qquad
L_-:=\mathcal{B}\setminus H_-.
\]
Since $C_+$ is a Hamiltonian circle, there exists a co-Hamiltonian sequence whose set of removed faces is $L_+$.
Likewise, there exists a co-Hamiltonian sequence whose set of removed faces is $L_-$.
Choose such sequences and write them as
\[
\mathcal{F}_1=(F_1,\dots,F_k)
\qquad\text{and}\qquad
\mathcal{F}_2=(F'_1,\dots,F'_\ell),
\]
so that $L_1:=\{F_1,\dots,F_k\}=L_+$ and $L_2:=\{F'_1,\dots,F'_\ell\}=L_-$.
By Lemma~\ref{facegraph},
\[
\sigma(C_+)=\sigma(H_+)=\sigma(\mathcal{B})\sigma(L_+),
\qquad
\sigma(C_-)=\sigma(H_-)=\sigma(\mathcal{B})\sigma(L_-).
\]
Multiplying gives
\[
\sigma(C_+)\sigma(C_-)=\sigma(L_+)\sigma(L_-).
\]
Since $\sigma(C_+)=+1$ and $\sigma(C_-)=-1$, we have $\sigma(C_+)\sigma(C_-)=-1$, hence
\[
\sigma(L_+)\sigma(L_-)=-1
\qquad\Longleftrightarrow\qquad
\sigma(L_+)=-\sigma(L_-).
\]
Equivalently,
\[
\prod_{i=1}^{k} \sigma(F_i)=-\prod_{j=1}^{\ell} \sigma(F'_j),
\]
as desired.
\end{proof}

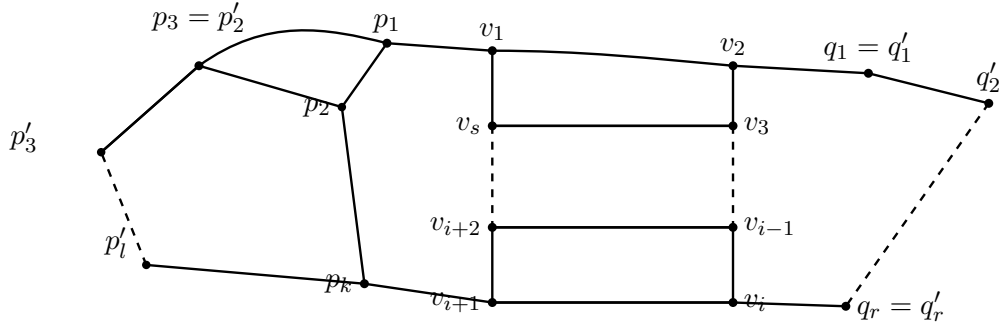
\begin{figure}[H]
\begin{tikzpicture}[scale=1, every node/.style={font=\small}]

\coordinate (p2) at (-3.9,3.05);

\coordinate (p12) at (-2,2.5);
\coordinate (p1) at (-1.4,3.35);
\coordinate (v1) at (0,3.25);
\coordinate (v2) at (3.2,3.05);
\coordinate (q1) at (5.0,2.95);
\coordinate (q2) at (6.6,2.55);   
\coordinate (tB) at (7.6,1.7);

\node[left] at (-2,2.5) {$p_2$};
\fill (p12) circle (1.6pt);
\draw[line width=0.9pt]
  (p2) .. controls (-3.0,3.7) and (-2.2,3.55) .. (p1)
  -- (v1) .. controls (1.2,3.25) and (2.1,3.15) .. (v2)
  -- (q1) -- (q2) ;
\draw[line width=0.9pt] (p1)--(p12);
\coordinate (pK) at (-1.7,0.15);
\draw[ line width=0.9pt] (pK)--(p12);

\coordinate (lp) at (-5.2,1.9);
\coordinate (lb) at (-4.6,0.4);

\fill (lp) circle (1.6pt);
\fill (lb) circle (1.6pt);
\node[left] at (-4.7,0.7) {$p_l'$};
\node[left] at (-5.9,2.1) {$p_3'$};   
\node[above] at (-3.9,3.35) {$p_3=p_2'$}; 
\draw[ line width=0.9pt] (lp)--(p2);
\draw[ line width=0.9pt] (lp)--(p2);
\draw[ line width=0.9pt] (lb)--(pK);
\draw[ line width=0.9pt] (p2)--(p12);

\draw[ dashed,line width=0.9pt] (lb)--(lp);

\coordinate (vh) at (0,2.25);
\coordinate (v3) at (3.2,2.25);

\draw[line width=0.9pt] (v1)--(vh)--(v3)--(v2);
\draw[line width=0.9pt] (vh)--(v3);

\coordinate (vL2) at (0,0.8);
\coordinate (vRim1) at (3.2,0.8);

\draw[dashed, line width=0.9pt] (vh) -- (vL2);
\draw[dashed, line width=0.9pt] (v3) -- (vRim1);

\coordinate (vIp2) at (0,0.9);
\coordinate (vIp1) at (0,-0.1);
\coordinate (vIm1) at (3.2,0.9);
\coordinate (vI)   at (3.2,-0.1);

\draw[line width=0.9pt] (vIp2)--(vIp1)--(vI)--(vIm1)--(vIp2);
\draw[line width=0.9pt] (vIp2)--(vIm1);

\coordinate (pK) at (-1.7,0.15);
\coordinate (qk) at (4.7,-0.15);

\draw[line width=0.9pt] (pK)--(vIp1)--(vI)--(qk);

\foreach \pt in {p2,p1,v1,v2,q1,q2,vh,v3,vIp2,vIp1,vIm1,vI,pK,qk}{
  \fill (\pt) circle (1.7pt);
}

\node[above] at (p1) {$p_1$};
\node[above] at (v1) {$v_1$};
\node[above] at (v2) {$v_2$};
\node[above] at (q1) {$q_1=q'_1$};
\node[above] at (q2) {$q'_2$};

\node[left]  at (vh) {$v_s$};
\node[right] at (v3) {$v_3$};

\node[left]  at (vIp2) {$v_{i+2}$};
\node[left]  at (vIp1) {$v_{i+1}$};
\node[right] at (vIm1) {$v_{i-1}$};
\node[right] at (vI)   {$v_i$};

\node[left]  at (pK) {$p_k$};
\node[right] at (qk) {$q_r=q'_r$};
\draw[dashed,line width=0.9pt] (q2)--(qk);

\end{tikzpicture}

\caption{The graph $G$ with the central ladder region used in the setup of Theorem~\ref{thm:opposite-signs-give-two-ham}.}

\label{fig:placeholder9}
\end{figure}


Let $\Sigma=(G,\sigma)$ be a simple plane signed graph (See Figure \ref{fig:placeholder9}) whose outer face boundary is the circle
\[
C_{\mathrm{out}}
:=v_1v_2q'_1q'_2\cdots q_r'\,v_i v_{i+1}\,p_k\,p'_{\ell}p'_{\ell-1}\cdots p'_3p'_2p'_1\,v_1.
\]
Assume that $\Sigma$ contains a circle
\[
C:=v_1v_2v_3\cdots v_{i-1}v_i v_{i+1}v_{i+2}\cdots v_s v_1
\]
such that the closed disk bounded by $C$ contains no vertices of $\Sigma$ in its interior.

Define the circle
\[
L
:=v_1v_s v_{s-1}\cdots v_{i+2}v_{i+1}\,p_k\,p'_{\ell}p'_{\ell-1}\cdots p'_3p'_2p'_1\,v_1.
\]
Assume that the edges of the path
\[
P_L:=p_1v_1v_s v_{s-1}\cdots v_{i+2}v_{i+1}p_k
\]
are fixed, i.e., none of the edges of $P_L$ is allowed to be deleted.
Suppose that there exists an edge set $E_L\subseteq E(\Sigma)$ disjoint from $E(P_L)$ such that, after deleting $E_L$,
\begin{enumerate}
\item every vertex  inside $L$ becomes incident with the outer face, and
\item the resulting graph $\Sigma-E_L$ is still $2$-connected.
\end{enumerate}
In particular, in $\Sigma-E_L$ there is a new outerface path from $p_1$ to $p_k$.

Assume symmetrically that the analogous statement holds on the right side: after deleting some edge set
$E_R$, while keeping the boundary path $q_1v_2v_3\cdots v_{i-1}v_iq_r$ fixed, every vertex in the corresponding
right region becomes incident with the outer face, the resulting graph remains $2$-connected, and in the end the circle
\[
R:=v_2q_1q_2\cdots q_r\,v_i v_{i-1}\cdots v_3v_2
\]
bounds a closed disk containing no vertices of $G$ in its interior.

\begin{theorem}\label{thm:opposite-signs-give-two-ham}
Under the assumptions above, let
\[
C_1:=v_1v_2v_3v_sv_1
\qquad\text{and}\qquad
C_2:=v_{i-1}v_iv_{i+1}v_{i+2}v_{i-1}.
\]
If $\sigma(C_1)\neq \sigma(C_2)$, then $\Sigma$ contains two Hamiltonian circles of opposite sign.
\end{theorem}

\begin{proof}
Let $\Sigma=(G,\sigma)$ satisfy all hypotheses, and keep the notation $E_L,E_R,L,R$ from the setup.

\medskip
\noindent\textbf{Step 1: Reduce to the released graph.}
Let
\[
\Sigma':=\Sigma-(E_L\cup E_R).
\]
By assumption, $\Sigma'$ is still $2$-connected. Moreover, the left and right release assumptions imply that every
vertex that was  inside $L$ (respectively, inside $R$) becomes incident with the outer face of
$\Sigma'$. In particular, all vertices on the $p$-chain ($:=p_1p_2...p_k$) and $q$-chain lie on the outer boundary of $\Sigma'$.
Only vertices in the ladder region (Figure~\ref{fig:placeholder9}) may remain non-outer.

\medskip
\noindent\textbf{Step 2: Define $H_1$ and $H_2$.}
Consider the Hamiltonian circle $H_1$ given  by
\[
H_1 := v_1v_s\cdots v_{i+2}v_{i-1}\cdots v_3v_2q_1q_2\cdots q_rv_iv_{i+1}p_k\cdots p_2p_1v_1.
\]
Thus $H_1$ traverses every vertex exactly once.

Now define $H_2$ to be the circle obtained from $H_1$ by toggling along $C_1$ and $C_2$, that is,
\[
E(H_2):=E(H_1)\oplus E(C_1)\oplus E(C_2),
\]
where $\oplus$ denotes symmetric difference.

\medskip
Observe that $H_2$ is spanning and $2$-regular, that is, it contains all vertices and every vertex has degree $2$ in $H_2$.
Moreover, the toggling operation only reroutes the traversal locally and does not disconnect the circle.
Hence $H_2$ is a Hamiltonian circle.

\medskip
\noindent\textbf{Step 3: Compare the signs.}
Since $E(H_2)=E(H_1)\oplus E(C_1)\oplus E(C_2)$ , we have
\[
\sigma(H_2)=\sigma(H_1)\,\sigma(C_1)\,\sigma(C_2).
\]
Because each circle sign is $\pm1$, the hypothesis $\sigma(C_1)\neq \sigma(C_2)$ implies
$\sigma(C_1)\sigma(C_2)=-1$, and hence $\sigma(H_2)=-\sigma(H_1)$.

Therefore $\Sigma$ contains two Hamiltonian circles of opposite sign, namely $H_1$ and $H_2$.

\end{proof}

We now describe a second local configuration that forces the existence of Hamiltonian circles of both signs.
\begin{figure}[H]
  \begin{tikzpicture}[scale=1, every node/.style={font=\small}]
\coordinate (p2) at (-5.0,2.0);
\coordinate (p1) at (-3.6,2.0);
\coordinate (v1) at (-1.8,2.0);
\coordinate (v2) at ( 1.2,2.0);
\coordinate (q1) at ( 2.4,2.0);
\coordinate (q2) at ( 3.6,1.9);

\draw[line width=0.9pt] (p2)--(p1)--(v1)--(v2)--(q1)--(q2);

\coordinate (pk) at (-3.6,0.0);
\coordinate (v4) at (-1.4,0.0);
\coordinate (v3) at ( 1.2,0.0);
\coordinate (qk) at ( 2.8,0.0);

\draw[line width=0.9pt] (pk)--(v4)--(v3)--(qk);

\coordinate (rL) at (-2.4,1.0);
\coordinate (r4) at ( 2.1,1.0);

\draw[line width=0.9pt]
  (v1) .. controls (-1.9,1.55) and (-2.1,1.25) .. (rL)
       .. controls (-2.0,0.75) and (-1.7,0.35) .. (v4);

\draw[line width=0.9pt]
  (v2) .. controls ( 1.0,1.55) and ( 1.4,1.25) .. (r4)
       .. controls ( 1.4,0.75) and ( 1.1,0.35) .. (v3);

\draw[dotted, line width=1.0pt] (rL) -- (r4);

\coordinate (m1) at (-1.2,1.0);
\coordinate (m2) at (-0.2,1.0);
\coordinate (m3) at ( 0.9,1.0);


\foreach \pt in {p2,p1,v1,v2,q1,q2,pk,v4,v3,qk,rL,r4,m1,m2,m3}{
  \fill (\pt) circle (1.7pt);
}

\node[above] at (p2) {$p_2$};
\node[above] at (p1) {$p_1=p'_1$};
\node[above] at (v1) {$v_1$};
\node[above] at (v2) {$v_2$};
\node[above] at (q1) {$q_1=q'_1$};
\node[above] at (q2) {$q_2$};

\node[below] at (pk) {$p_k=p'_k$};
\node[below] at (v4) {$v_4$};
\node[below] at (v3) {$v_3$};
\node[below] at (qk) {$q_l=q'_l$};

\node[left]  at (rL) {$r_1$};
\node[right] at (r4) {$r_t$};
\draw[dashed,line width=0.9pt] (p2)--(pk);
\draw[dashed,line width=0.9pt] (q2)--(qk);
\end{tikzpicture}

  \caption{The resulting graph $\Sigma_0$ after removing the edge sets $E_L$ and $E_R$ from $\Sigma$.}

    \label{fig:placeholder}
\end{figure}
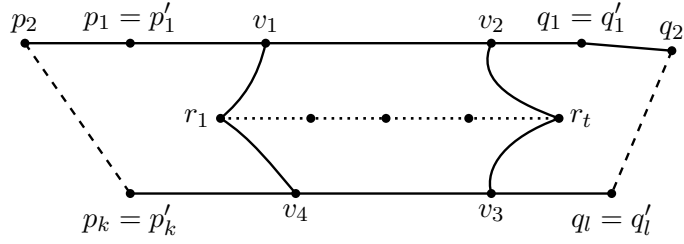

Let $\Sigma=(G,\sigma)$ be a simple plane signed graph with outer circle
\[
C_{\mathrm{out}}
:=v_1v_2q'_1q'_2\cdots q'_\ell v_3v_4p'_kp'_{k-1}\cdots p'_2p'_1v_1.
\]

Assume that $\Sigma$ contains an  hexagon (it is possible that $r_1=r_t$)
\[
H:=v_1v_2r_tv_3v_4r_1v_1.
\]
Assume moreover that every vertex in the interior of $H$ lies on a path
\[
R:=r_1r_2\cdots r_t
\]
embedded inside $H$; in particular, $r_1$ and $r_t$ lie on $H$, while all other vertices
$r_2,\dots,r_{t-1}$ lie  in the interior of $H$.

Consider the circle
\[
C'_L:=v_1r_1v_4p'_kp'_{k-1}\cdots p'_2p'_1v_1.
\]
Assume that the four  edges
\[
p_1v_1,\quad v_1r_1,\quad r_1v_4,\quad v_4p_k, (\text{here}, p_1=p'_1,p_k=p'_k)
\]
are fixed, that is, they are not allowed to be deleted.
Suppose that there exists a set (or sequence) of edges
$E_L\subseteq E(G)$ such that deleting $E_L$ releases all vertices
strictly inside $C'_L$ to the outer face, and the resulting graph
$\Sigma-E_L$ is still $2$-connected. After deleting the edge set $E_L$, the resulting graph $\Sigma-E_L$ contains the circle
\[
C_L:=v_1r_1v_4p_kp_{k-1}\cdots p_2p_1v_1.
\]

Similarly, consider the circle
\[
C'_R:=v_2q'_1q'_2\cdots q'_\ell v_3r_tv_2.
\]
Assume that there exists a set (or sequence) of edges
$E_R\subseteq E(G)$ such that deleting $E_R$ releases all vertices
strictly inside $C'_R$ to the outer face, and the resulting graph
$\Sigma-E_R$ is still $2$-connected. At the end of the right-side deletion process, the resulting graph contains the circle
\[
C_R:=v_2q_1q_2\cdots q_\ell v_3r_tv_2.
\]

\begin{theorem}\label{thm:two-opposite-ribbon-cycles-give-both-signs}
Let $\Sigma=(G,\sigma)$  be a simple plane signed graph satisfying all assumptions stated above.
In particular, $\Sigma$ has outer circle
$C_{\mathrm{out}}$ that contains the  hexagon
$H:=v_1v_2r_4v_3v_4r_1v_1,$
and every vertex in the interior of $H$ lies on a path $R=r_1r_2\cdots r_t$ embedded inside $H$.
Assume moreover that the left and right release conditions hold for the circles $C'_L$ and $C'_R$, and after releasing edge we obtain $C_L$ and $C_R$. If two circles $C_1:=v_1v_2r_t r_{t-1}\cdots r_1v_1$ and $C_2:=v_3v_4r_1 r_2\cdots r_tv_3$ have different signs, then $\Sigma$ contains both a positive Hamiltonian circle and a negative
Hamiltonian circle.
\end{theorem}

\begin{proof}
Let $\Sigma=(G,\sigma)$ satisfy the assumptions stated above.  By hypothesis, there exist edge sets
$E_L,E_R\subseteq E(G)$ such that deleting $E_L$ releases all vertices  inside $C'_L$ to the outer face
and $\Sigma-E_L$ remains $2$-connected, and similarly deleting $E_R$ releases all vertices  inside $C'_R$
to the outer face and $\Sigma-E_R$ remains $2$-connected.

Let
\[
\Sigma_0 := \Sigma-(E_L\cup E_R).
\]
Then $\Sigma_0$ is still $2$-connected (See Figure \ref{fig:placeholder}).  Moreover, by the two release conditions, every vertex of $G$ outside the
``central'' region bounded by the hexagon is now incident with the outer face of $\Sigma_0$; hence the only vertices
that can remain in the interior of the central region are exactly the vertices on the path
$R=r_1r_2\cdots r_t$.

In $\Sigma_0$ consider the following two circles:
\[
H_1 := v_1v_2q_1q_2\cdots q_\ell v_3r_t r_{t-1}\cdots r_1v_4p_kp_{k-1}\cdots p_2p_1v_1,
\]
\[
H_2 := v_1r_1r_2\cdots r_tv_2q_1q_2\cdots q_\ell v_3v_4p_kp_{k-1}\cdots p_2p_1v_1.
\]
Both $H_1$ and $H_2$ traverse every vertex of $G$ exactly once.  Therefore, $H_1$ and $H_2$ are Hamiltonian circles of $\Sigma_0$.
Since $\Sigma_0$ is obtained from $\Sigma$ by deleting edges only, these circles are also Hamiltonian circles of $\Sigma$.

It remains to compare their signs.  Let
\[
Q := v_2q_1q_2\cdots q_\ell v_3
\quad\text{and}\quad
P := v_4p_kp_{k-1}\cdots p_2p_1v_1,
\]
and let $R$ also denote the path $r_1r_2\cdots r_t$.
Write $\sigma(Q)$, $\sigma(P)$, and $\sigma(R)$ for the products of the edge-signs along these paths.
Then, by inspection,
\[
\sigma(H_1)=\sigma(P)\,\sigma(Q)\,\sigma(R)\,\sigma(v_1v_2)\,\sigma(v_3r_t)\,\sigma(v_4r_1),
\]
\[
\sigma(H_2)=\sigma(P)\,\sigma(Q)\,\sigma(R)\,\sigma(v_1r_1)\,\sigma(v_2r_t)\,\sigma(v_3v_4).
\]
On the other hand, for the two circles in the statement,
\[
C_1=v_1v_2r_tr_{t-1}\cdots r_1v_1
\quad\text{and}\quad
C_2=v_3v_4r_1r_2\cdots r_tv_3,
\]
we have
\[
\sigma(C_1)=\sigma(R)\,\sigma(v_1v_2)\,\sigma(v_2r_t)\,\sigma(v_1r_1),
\qquad
\sigma(C_2)=\sigma(R)\,\sigma(v_3v_4)\,\sigma(v_4r_1)\,\sigma(v_3r_t).
\]
Therefore
\[
\frac{\sigma(H_1)}{\sigma(H_2)}
=
\frac{\sigma(v_1v_2)\,\sigma(v_3r_t)\,\sigma(v_4r_1)}
     {\sigma(v_1r_1)\,\sigma(v_2r_t)\,\sigma(v_3v_4)}
=
\frac{\sigma(C_2)}{\sigma(C_1)}.
\]
In particular, $\sigma(H_1)=\sigma(H_2)$ if and only if $\sigma(C_1)=\sigma(C_2)$.
By hypothesis, $\sigma(C_1)\neq\sigma(C_2)$, so $\sigma(H_1)\neq\sigma(H_2)$.
Since each circle-sign is $\pm 1$, this implies $\sigma(H_1)=-\sigma(H_2)$.

Hence $\Sigma$ contains one positive Hamiltonian circle and one negative Hamiltonian circle.
\end{proof}

\newpage

\section{Acknowledgment}

The author thanks Professor Thomas Zaslavsky for his invaluable suggestions and recommendations during the preparation of this draft.


\end{document}